\documentclass[11pt,draft,reqno]{amsart}

\usepackage{amssymb} 
\usepackage{dsfont} 
\usepackage{mathrsfs} 
\usepackage{stmaryrd} 
\usepackage{color}

\DeclareMathAlphabet{\mathpzc}{OT1}{pzc}{m}{it} 

\newtheorem{Thm}{Theorem}[section]
\newtheorem{Cor}{Corollary}[section]
\newtheorem{Lem}{Lemma}[section]
\newtheorem{Prop}{Proposition}[section]

\newtheorem{Def}{Definition}[section]

\theoremstyle{definition}
\newtheorem{Rem}{Remark}[section]

\theoremstyle{definition}
\newtheorem*{Ack}{Acknowledgments}

\makeatletter
\@addtoreset{equation}{section}
\makeatother

\newcommand{\AC}{{\textsl{AC}\hspace{0.17ex}}} 
\newcommand\setmeno{\!\smallsetminus\!} 
\newcommand\function{\longrightarrow} 
\newcommand\indicator{\mathds{1}} 
 \newcommand\en{\mathbb{N}} 
\newcommand\ar{\mathbb{R}} 
\providecommand{\clint}[1]{\hspace{0.045ex}\left[#1\right]} 
\providecommand{\clsxint}[1]{\hspace{0.1ex}\left[#1\right[\hspace{0.15ex}} 
\providecommand{\opint}[1]{\hspace{0.15ex}\left]#1\right[\hspace{0.15ex}} 

\newcommand{\lebsets}{\mathscr{M}} 
\renewcommand{\L}{{\textsl{L}\hspace{0.17ex}}} 
\newcommand\leb{\mathpzc{L}} 
\DeclareMathOperator{\de}{d \! \hspace{0.2ex}} 
\newcommand{\Step}{{\textsl{St}\hspace{0.17ex}}} 

\newcommand\X{\textsl{X}\hspace{0.21ex}} 
\renewcommand\d{\textsl{d}} 
\newcommand\As{\textsl{A}\hspace{0.21ex}} 
\DeclareMathOperator{\Int}{int} 
\DeclareMathOperator{\Cl}{cl} 
\newcommand\Y{\textsl{Y}\hspace{0.21ex}} 
\renewcommand\S{\textsl{S}\hspace{0.21ex}} 

\DeclareMathOperator{\cont}{Cont}  
\DeclareMathOperator{\discont}{Discont}  
\newcommand{\Czero}{{\textsl{C}\hspace{0.18ex}}} 
\DeclareMathOperator{\Lipcost}{Lip} 
\newcommand{\Lip}{{\textsl{Lip}\hspace{0.15ex}}} 

\newcommand\E{\textsl{E}\hspace{0.21ex}} 
\newcommand\norm[2]{\Vert #1\Vert_{#2}} 

\renewcommand\H{\mathcal{H}} 
\newcommand\duality[2]{\langle #1,#2 \rangle} 
\newcommand{\Conv}{\mathscr{C}} 
\newcommand\K{\mathcal{K}} 
\DeclareMathOperator{\Proj}{Proj} 

\newcommand\hausd{\mathpzc{H}} 
\newcommand\A{\mathcal{A}} 
\newcommand\B{\mathcal{B}} 
\newcommand\C{\mathcal{C}} 
\newcommand\G{\mathcal{G}} 

\newcommand\Z{\mathcal{Z}} 

\newcommand{\BV}{{\textsl{BV}\hspace{0.17ex}}} 
\newcommand{\W}{{\textsl{W}\hspace{0.17ex}}} 
\DeclareMathOperator{\V}{V} 
\DeclareMathOperator{\pV}{V} 
\renewcommand\r{\textsl{r}} 
\newcommand\vartot[1]{\!\left\bracevert\! #1 \!\right\bracevert\!} 
\DeclareMathOperator{\D}{D\!} 

\newcommand{\ftilde}{\widetilde{f}} 
\newcommand{\Ctilde}{\widetilde{\mathcal{C}}} 
\newcommand{\utilde}{\widetilde{u}} 
\newcommand{\yhat}{\hat{y}} 

\renewcommand{\P}{{\mathsf{P}}} 
\newcommand{\M}{{\mathsf{M}}} 
\newcommand{\R}{{\mathsf{R}}} 

\providecommand{\cldxint}[1]{\hspace{0.15ex}\left]#1\right]} 
\renewcommand\l{\textsl{l}} 

\definecolor{blu}{rgb}{0.1,0.1,1}

\definecolor{green}{rgb}{0.0, 0.5, 0.0}

\definecolor{marr}{rgb}{0.63, 0.47, 0.35}

\begin{document}


\title[Multidimensional play]{Multidimensional play operators \\ with arbitrary BV inputs}

\author{Vincenzo Recupero}

\address{\textbf{Vincenzo Recupero} \\
        Dipartimento di Scienze Matematiche \\ 
        Politecnico di Torino \\
        Corso Duca degli Abruzzi 24 \\ 
        I-10129 Torino \\ 
        Italy. \newline
        {\rm E-mail address:}
        {\tt vincenzo.recupero@polito.it}}




\begin{abstract}
In this paper we provide an integral variational formulation for a vector play operator where the inputs are allowed to be 
arbitrary functions with (pointwise) bounded variation, not necessarily left or right continuous. We prove that this problem admits a unique solution, and we show that in the left continuous and right continuous cases it reduces to the well known existing formulations.

\end{abstract}


\maketitle


\thispagestyle{empty}


\section{Introduction}

A rigorous explicit mathematical treatment of hysteresis probably started in the pioneering paper \cite{Bou71}, and in the monograph \cite{KraPok89}. In particular the authors of \cite{KraPok89} presented many original theorems together with previous results published in papers dated back to 1970. This monograph stimulated an extensive research which is testified for instance in the books \cite{Vis94, BroSpr96, Kre97, MieRou15} and in the references therein. 

One of the simplest models of hysteresis is the so-called \emph{play} operator. In a fair general framework it can be described in the following way. Let $\H$ be a real Hilbert space with inner product
$\duality{\cdot}{\cdot}$, let $\Z \subseteq \H$ be a closed convex subset containing the zero vector, and let $T > 0$ be a final time of the evolution considered. If $\W^{1,1}(\clint{0,T};\H)$ denotes the space of $\H$-valued absolutely continuous maps, then the \emph{(vector) play operator} is the operator $\P : \W^{1,1}(0,T;\H) \times \Z \function \W^{1,1}(0,T;\H)$ assigning to every 
couple $(u, z_0) \in \W^{1,1}(0,T;\H) \times \Z$ the only function $\P(u,z_0) := y \in \W^{1,1}(0,T;\H)$ satisfying the following three conditions
\begin{alignat}{3}
  & u(t) - y(t) \in \Z & \qquad & \forall t \in \clint{0,T}, \label{intro constr for P} \\
  & \duality{z - u(t) + y(t)}{y'(t)} \le 0 & \qquad & \forall z \in \Z, \quad \text{for $\leb^1$-a.e. $t \in \clint{0,T}$}, 
     \label{intro ineq for P} \\
  & u(0) - y(0) = z_{0} \in \Z,   \label{intro i.c. for P}
\end{alignat}
where $\leb^1$ denotes the one-dimensional Lebesgue measure and $y'$ is the time derivative of $y$ (in the following section we will recall all the precise definitions and theorems needed in the paper).
The play operator can also be considered as a model for a ``strain $\mapsto$ stress'' relation $u \mapsto x := u - y$ where the boundary of $\Z$ represents the so called yield surface: as long as $x$ remains in the interior $\Z$ the system has an elastic behavior; when $x$ touches the boundary of $\Z$ the system becomes plastic and a deformation occurs.

It is well known that $\P$ is particular cases of a hysteresis operator, i.e. an operator 
$\R : \W^{1,1}(0,T;\H) \times \Z \function \W^{1,1}(0,T;\H)$ which is  
\emph{causal}: 
\begin{equation}
  u = v \ \text{on $\clint{0,t}$} \quad \Longrightarrow \quad \R(u, z_0)(t) = \R(v, z_0)(t),
\end{equation}
and \emph{rate independent}:
\begin{equation}
  \R(u \circ \phi, z_0) = \R(u, z_0) \circ \phi 
\end{equation}
whenever $t \in \clint{0,T}$, $u, v \in \W^{1,1}(\clint{0,T};\H)$, and $\phi : \clint{0,T} \function \clint{0,T}$ is a surjective
absolutely continuous functions. Note that when $\H = \ar$, rate independence is the property that allows to represent the 
hysteresis loops in a plane $(u,y) = (u,\R(u,z_0))$ with no reference to time. The suggestive terms \emph{input} and \emph{output} are often used to refer to $u$ and $\R(u,z_0)$ respectively.
 
The play operator can be considered as a first order evolution problem in the unknown $y$, indeed we can rewrite \eqref{intro ineq for P} in the following way:
\begin{equation}
  y'(t) \in -\partial I_{u(t) - \Z}(y(t)) \qquad \text{for $\leb^1$-a.e. $t \in \clint{0,T}$}, \label{gradient flow for y - intro} 
\end{equation}
$\partial I_\Z$ being the subdifferential of the indicator function $I_\Z$:  $I_{\Z}(x) := 0$ if $x \in \Z$, $I_{\Z}(x) := \infty$ otherwise:
\begin{equation}
  \partial I_\Z(x) := \{v \in \H\ : \ \duality{v}{z-x} \le 0\ \forall z \in \Z\},
\end{equation}
essentially the set of all subtangents of $I_\Z$ at $x$.
thus the solution operator $\P$ can be considered as a particular case of \emph{sweeping process}, that is a generalized time dependent gradient flow of the type
\begin{equation}
  y'(t) \in - \partial I_{\C(t)}(y(t)) \qquad \text{for $\leb^1$-a.e. $t \in \clint{0,T}$}, \label{intro sweeping process}
\end{equation}
which includes the play operator by considering 
\begin{equation}
  \C(t) = \C_u(t) := u(t) - \Z, \qquad \in \clint{0,T}
\end{equation}
(for the sweeping processes we refer the reader, e.g., to \cite{Mor71, Mor76, Mor77, Mon93, AdlHadThi14, Rec16a}).

When $\H = \ar$ in the monographs \cite[Section 6.7]{KraPok89} and in \cite[Definition 2.3.13]{BroSpr96}
the play operator is extended to the space of functions of bounded variation $\BV(\clint{0,T};\ar)$ by using 
the property of rate independence: essentially they `fill in' the jumps with segments traversed with an infinite speed so that 
they can apply the operators in the continuous case. When the dimension of $\H$ is greater that 1, as observed in \cite{KreLau02}, ``this procedure turns out to be trajectory-dependent which makes the analysis difficult even if we restrict to some canonical (the shortest, say) trajectory filling in the jumps''. Therefore they follow another method which works in the space $\BV^\l(\clint{0,T};\H)$ of left continuous functions of bounded variation: it consists in defining the play operator $\P : \BV^\l(\clint{0,T};\H) \function \BV^\l(\clint{0,T};\H)$ as the solution operator which assigns to 
$u \in \BV^\l(\clint{0,T};\H)$ and $z_0 \in \Z$ the unique function $\P(u,z_0) := y \in \BV^\l(\clint{0,T};\H)$ satisfying the conditions
\begin{alignat}{3}
  & u(t) - y(t) \in \Z & \qquad & \forall t \in \clint{0,T}, \label{intro constr for P-BV} \\
  & \int_0^T\duality{z(t) - u(t+) + y(t)}{\de y(t)} \le 0 & \qquad & \forall z \in \BV(\clint{0,T};\H),\ z(\clint{0,T}) \subseteq \Z, 
     \label{intro ineq for P-BV} \\
  & u(0) - y(0) = z_{0} \in \Z, \label{intro i.c. for P-BV}
\end{alignat}
where the integral in \eqref{intro ineq for P-BV} is the integral in the sense of Young (see \cite[Section 3]{KreLau02}). Notice that this integral can also be written as a Lebesgue integral with respect to the measure $\D y$, the distributional 
measure of $y$, in the following way:
\begin{equation}
  \int_0^T\duality{z(t) - u(t+) + y(t+)}{\de \D y(t)} \le 0  \qquad  \forall z \in \BV(\clint{0,T};\H),\ z(\clint{0,T}) \subseteq \Z,
\end{equation}
as shown in \cite[Section A.4]{Rec11a} and in \cite[Theorem 3.2]{KopRec16}, or in \cite[Lemma 4.1]{RecSan18} in the more general setting of sweeping processes. We also mention another approach to the vector play operator with discontinuous inputs based on the papers \cite{Kle12, Kle14, Kle15, Kle16}.

In the present paper we follow a different procedure and we provide a multidimensional generalization of the methods of Krasnosel'ski\v{i} and Pokrovski\v{i} (\cite[Section 6.7, p. 56]{KraPok89}) and of Brokate and Sprekels 
(\cite[p. 51]{BroSpr96}), where the scalar play operator is extended  to the whole space of functions of bounded variation (not necessarily left or right continuous), and we obtain, as a byproduct, an explicit formulation (see \eqref{BV constraint for P-intro}--\eqref{BV i.c. for P-intro} below) and an existence/uniqueness result for the vector play operator acting on the whole space $\BV(\clint{0,T};\H)$, not only on its proper subset $\BV^\l(\clint{0,T};\H)$. 
In this regard let us observe that in \cite{KraPok89, BroSpr96} the scalar play is extended to 
$\BV(\clint{0,T};\ar)$, but no explicit formulation for the resulting extended operator is provided. Let us also notice that in the paper \cite{Mor77} Moreau studied the existence and uniqueness of solutions of (vectorial) sweeping processes with $\BV$ inputs which are not necessarily left or right continuous: he defined and obtained these solutions as a uniform limit of a sequence of solutions of properly discretized problems, but he provided an explicit formulation of the related problem as a differential inclusion only in the case of right continuous inputs. 

The motivations to deal with arbitrary $\BV$ inputs rather than left or right continuous ones are various. First of all
let us recall that $u(t)$ can be considered as an external loading of the physical system modeled by 
\eqref{intro constr for P}--\eqref{intro i.c. for P}, therefore if we deal with 
a loading that at a certain time $t_0$ moves very quickly from a position $x_{0,1}$ to a position $x_{0,2}$ and then possibly to a third position $x_{0,3}$, it is very natural to model such loading by assuming that $u(t_0-) = x_{0,1}$, 
$u(t_0) = x_{0,2}$, and $u(t_0+) = x_{0,3}$, and this corresponds to the approach of \cite{KraPok89, BroSpr96} in the scalar case. Of course one can argue if the resulting formulation (or even \eqref{intro constr for P}--\eqref{intro i.c. for P} in the left continuous case) is a suitable modelization of these kind of phenomena also in the vectorial case (for instance by performing simulations and measurements of suitable experiments), anyhow we are able show that this natural generalization is uniquely solvable.
Another motivation can be found in the introduction of the Moreau's paper \cite{Mor77} where the following hydrodynamical system is considered: there $\H = \ar^2$ and $\C(t) \subseteq \H$ is given for any $t$ belonging to the interval $\clint{0,T}$ which is interpreted as a segment of the vertical axis of the three dimensional physical space oriented downward. In this way $G = \{(x,t) \in \H \times \clint{0,T}\ :\ x \in \C(t)\}$ represents a solid cavity which in the case of the vectorial play operator is a sort of tube whose horizontal sections are a translation of the set $\Z \subseteq \ar^2$, since $\C(t) = u(t) - \Z$ for every $t \in \clint{0,T}$. According to \eqref{intro sweeping process}, it follows that graph of the solution $y(t)$ is ``a tiny stationary waterstream falling down the cavity''.  Therefore ``any arc of this stream which happens to be loose from the cavity wall is rectilinear and vertical'', while ``when water is running over the wall, it describes a line orthogonal to the level curves of the wall surface, i.e., a line of steepest descent; this agrees with hydrodynamics under the simplifying assumption that inertia may be neglected comparatively to friction and gravity''. A discontinuity at $t = t_0$ represents the fact that at a certain level $t_0$ the cavity has been moved from a position 
$\C(t_0) = \C(t_0-) = u(t_0-) - \Z$ to a position $\C(t_0+) = u(t_0+) - \Z$, due to some intrinsic constraints of the environment where the cavity has been inserted. The resulting model \eqref{intro sweeping process} (or \eqref{intro ineq for P-BV}) takes into account of the fact that, for instance, some mechanism exists or have been built in such a way that at the level $t_0$ the water is forced to follow the least distance to reach $\C(t_0+) = u(t_0+) - \Z$. Consequently, the fact that $u(t)$ is not necessarily left or right continuous allows to model a situation where at the level $t_0$ it is needed to bring the water from $\C(t_0-)$ to $\C(t_0+)$ through an intermediate region $\C(t_0)$.

Following \cite{KraPok89, BroSpr96}, the idea of the proof of our present paper is to fill in the jumps with a couple of suitable trajectory joining first $u(t-)$ with
$u(t)$, and then $u(t)$ with $u(t+)$ for every jump point $t \in \clint{0,T}$. Consequently we traverse these trajectories
with an infinite speed in order to obtain the extension of $\P$ to $\BV(\clint{0,T};\H)$.
The main difficulty in this procedure is the choice of the trajectories: if we consider a left continuous input and we fill in the jumps with segments, which seems the natural way, then in \cite{Rec11a} it is proved that the resulting operator is different from the play operator defined by \eqref{intro constr for P-BV}--\eqref{intro i.c. for P-BV} (a complete comparison between the two operators is performed in \cite{KreRec14a, KreRec14b} in the finite dimensional case). Thus another choice is in order and it seems that there is no chance to join the jumps in a canonical way that is intrinsic in the nature of the play operator and independent of the particular input. The idea to overcome this problem is to change the framework and interpret the play operator as a sweeping process, so that the inputs are now functions $\C$ having values in the metric space $\Conv_\H$ of nonempty closed convex sets of the form $\C(t) = u(t) - \Z$. Thus we have to find some convex-valued trajectory joining the sets $\C(t-) = u(t-) - \Z$ with $\C(t) = u(t) - \Z$ and $\C(t) = u(t) - \Z$ with $\C(t+) = u(t+) - \Z$. The choice has to be found among a family of geodesics in the space 
$\Conv_\H$, more precisely the proper geodesics are provided by 
\begin{equation}\label{G-}
  \G_{t-}(\sigma) := ( u(t-) - \Z + D_{\sigma\norm{u(t-)-u(t)}{}} ) \cap ( u(t) - \Z + D_{(1-\sigma)\norm{u(t-)-u(t)}{}} )
\end{equation}
and
\begin{equation}\label{G+}
  \G_{t+}(\sigma) := ( u(t) - \Z + D_{\sigma\norm{u(t)-u(t+)}{}} ) \cap ( u(t+) - \Z + D_{(1-\sigma)\norm{u(t)-u(t+)}{}} )
\end{equation}
where $D_\rho := \{x \in \H\ :\ \norm{x}{} \le \rho\}$ for $\rho > 0$.
Therefore our procedure essentially consists in reparametrizing by the arc length the function $\C_u(t) = u(t) - \Z$ and by filling in the jumps between $\C(t-) = u(t-) - \Z$ with $\C(t) = u(t) - \Z$ and $\C(t) = u(t) - \Z$ with $\C(t+) = u(t+) - \Z$
with the curves \eqref{G-} and \eqref{G+}. Then we take the solution of the corresponding sweeping process where
we essentially traverse these jumps with infinite velocity: this solution is the desired solution of the play operator
with the $\BV$ input and we prove that it is the the only function satisfying the following integral variational problem: 
given $u \in \BV(\clint{0,T};\H)$ there exists a unique function $y \in \BV(\clint{0,T};\H)$ such that 
\begin{alignat}{3}
  & u(t) - y(t) \in \Z & \qquad & \forall t \in \clint{0,T}, \label{BV constraint for P-intro} \\
  & \int_{\cont(u)} \duality{z(t) - u(t) + y(t)}{\de\D y(t)} \le 0 & \quad & 
                       \text{$\forall z \in \L^\infty(\clint{0,T};\H)$, $z(\clint{0,T}) \subseteq \Z$}, 
      \label{BV int ineq for P-intro} \\
  & u(t) - y(t) = \Proj_{\Z}(u(t) - y(t-)) & \qquad &  \forall t \in \discont(u) \setmeno \{0\}, \label{BV jump cond for P--intro} \\
  &  u(t+) - y(t+) = \Proj_{\Z}(u(t+) - y(t)) & \qquad &  \forall t \in \discont(u), 
       \label{BV jump cond for P+-intro} \\
  & u(0) - y(0) = z_0, \label{BV i.c. for P-intro}
\end{alignat}
where $\cont(u)$ and $\discont(u)$ denote respectively the continuity set and the discontinuity set of $u$.
Moreover in the left continuous case we prove that our generalized formulation coincide with the formulation 
\eqref{intro constr for P-BV}--\eqref{intro i.c. for P-BV}. Finally, let us also observe that when the inputs are continuous and with bounded variation there is no necessity to invoke the theory of sweeping processes and it is very easy to reduce the $\BV$ continuous case to the Lipschitz continuous case as showed in \cite{Rec08}.
The geodesics \eqref{G-} and \eqref{G+} belong to the general class of geodesics used in our paper \cite{Rec16a} in order to reduce the right continuous $\BV$ sweeping processes to the Lipschitz continuous case, but in the present paper the situation is different since in the case of arbitrary inputs in 
$\BV(\clint{0,T};\H)$
the well know formulation \eqref{intro constr for P-BV}--\eqref{intro i.c. for P-BV} is incorrect and we need to find a new formulation (namely \eqref{BV constraint for P-intro}--\eqref{BV i.c. for P-intro}) which is naturally inferred by means of our procedure by filling in the ``double'' jumps with the two geodesics 
\eqref{G-} and \eqref{G+} and by reducing to the Lipschitz continuous case.

Let us observe that it could be possible to infer the statement of our main result from the result of \cite{RecSan18} where sweeping processes with a prescribed behavior on jumps are dealt with. Nevertheless in \cite{RecSan18} our formulation for the play operator with arbitrary $\BV$-inputs is neither deduced nor explicitly written, moreover the proof we perform in the present paper is easier and geometrically more transparent than the analytic techniques used in \cite{RecSan18}, since we show the explicit way to connect the double jumps of 
$u(t) - \Z$, thereby displaying clearly what kind of dynamics happens in the jump points by rescaling the time and reducing to a Lipschitz continuous evolution thanks to the rate independence property of the play operator. This kind of proof also shows why the broader framework of sweeping process is somehow natural and how the choice of the geodesics plays a crucial role, since if we would remain in the classical framework of the play and if we would choose to connect the jumps of $u(t)$ with straight segments we would obtain another ``notion of solution'' which would not be consistent with the solution \eqref{intro constr for P-BV}--\eqref{intro i.c. for P-BV} of in the left continuous case.


\section{Preliminaries}\label{S:Preliminaries}

In this section we recall the main definitions and tools needed in the paper. The set of integers greater than or equal to 
$1$ will be denoted by $\en$. Given an interval $I$ of the real line $\ar$, if $\lebsets(I)$ indicates the family of Lebesgue measurable subsets of $I$, $\mu : \lebsets(I) \function \clint{0,\infty}$ is a measure, $p \in \clint{1,\infty}$, and if 
$\E$ is a Banach space, then the space of $\E$-valued functions which are $p$-integrable with respect to $\mu$ will be denoted by $\L^p(I, \mu; \E)$ or simply by $\L^p(\mu; \E)$. We do not identify two functions which are equal $\mu$-almost everywhere ($\mu$-a.e.). The one dimensional Lebesgue measure is denoted by $\leb^1$ which we assume to be complete. For the theory of integration of vector valued functions we refer, e.g., to \cite[Chapter VI]{Lan93}.

\subsection{Functions of bounded variation}

In this subsection we assume that 
\begin{equation}\label{X complete metric space}
  \text{$(\X,\d)$ is an extended complete metric space},
\end{equation}
i.e. $\X$ is a set and $\d : \X \times \X \function \clint{0,\infty}$ satisfies 
the usual axioms of a distance, but may take on the value $\infty$. The notion of completeness remains unchanged. 
The general topological notions of interior, closure and boundary of a subset $\Y \subseteq \X$ will be respectively denoted by $\Int(\Y)$, $\Cl(\Y)$ and $\partial \Y$. We also set $\d(x,\As) := \inf_{a \in \As} \d(x,a)$. If $(\Y,\d_\Y)$ is a metric space then the continuity set of a function $f : \Y \function \X$ is denoted by $\cont(f)$, while 
$\discont(f) := \Y \setmeno \cont(f)$. For $\S \subseteq \Y$ we write 
$\Lipcost(f,\S) := \sup\{\d(f(s),f(t))/\d_\Y(t,s)\ :\ s, t \in \S,\ s \neq t\}$, $\Lipcost(f) := \Lipcost(f,\Y)$, the Lipschitz constant of $f$, and $\Lip(\Y;\X) := \{f : \Y \function \X\ :\ \Lipcost(f) < \infty\}$, the set of $\X$-valued Lipschitz continuous functions on 
$\Y$.

\begin{Def}
Given an interval $I \subseteq \ar$, a function $f : I \function \X$, and a subinterval $J \subseteq I$, the 
\emph{(pointwise) variation of $f$ on $J$} is defined by
\begin{equation}\notag
  \pV(f,J) := 
  \sup\left\{
           \sum_{j=1}^{m} \d(f(t_{j-1}),f(t_{j}))\ :\ m \in \en,\ t_{j} \in J\ \forall j,\ t_{0} < \cdots < t_{m} 
         \right\}.
\end{equation}
If $\pV(f,I) < \infty$ we say that \emph{$f$ is of bounded variation on $I$} and we set 
$\BV(I;\X) := \{f : I \function \X\ :\ \pV(f,I) < \infty\}$.
\end{Def}

It is well known that the completeness of $\X$ implies that every $f \in \BV(I;\X)$ admits one-sided limits $f(t-), f(t+)$ at every point $t \in I$, with the convention that $f(\inf I-) := f(\inf I)$ if $\inf I \in I$, and $f(\sup I+) := f(\sup I)$ if $\sup I \in I$. Moreover $\discont(f)$ is at most countable. We set $\BV^\l(I;\X) := \{f \in \BV(I;\X)\ :\ f(t-) = f(t) \quad \forall t \in I\}$, 
$\BV^\r(I;\X) := \{f \in \BV(I;\X)\ :\ f(t) = f(t+) \quad \forall t \in I\}$, and if $I$ is bounded we have 
$\Lip(I;\X) \subseteq \BV(I;\X)$.

\subsection{Convex sets in Hilbert spaces}

Throughout the remainder of the paper we assume that
\begin{equation}\label{H-prel}
\begin{cases}
  \text{$\H$ is a real Hilbert space with inner product $(x,y) \longmapsto \duality{x}{y}$}, \\
  \norm{x}{} := \duality{x}{x}^{1/2},
\end{cases}
\end{equation}
and we endow $\H$ with the natural metric defined by $\d(x,y) := \norm{x-y}{}$, $x, y \in \H$. We set
\begin{equation}\notag
  \Conv_\H := \{\K \subseteq \H\ :\ \K \ \text{nonempty, closed and convex} \}.
\end{equation}
If  $\K \in \Conv_\H$ and $x \in \H$, then $\Proj_{\K}(x)$ is the projection on $\K$, i.e. $y = \Proj_\K(x)$ is the unique 
point such that $d(x,\K) = \norm{x-y}{}$, and it is also characterized by the two conditions
\begin{equation}\notag
  y \in \K, \quad \duality{x - y}{v - y} \le 0 \qquad \forall v \in \K.
\end{equation}
If $\K \in \Conv_\H$ and $x \in \K$, then $N_\K(x)$ denotes the \emph{(exterior) normal cone of $\K$ at $x$}:
\begin{equation}\label{normal cone}
  N_\K(x) := \{u \in \H\ :\ \duality{v - x}{u} \le 0\ \forall v \in \K\} = \Proj_\K^{-1}(x) - x.
\end{equation}
It is well known that the multivalued mapping $x \longmapsto N_\K(x)$ is \emph{monotone}, i.e. 
$\duality{u_1 - u_2}{x_1 - x_2} \ge 0$ whenever $x_j \in \K$, $u_j \in N_\K(x_j)$, $j = 1, 2$ (see, e.g., 
\cite[Exemple 2.8.2, p.46]{Bre73}). We endow the set $\Conv_{\H}$ with the Hausdorff distance. Here we recall the definition.

\begin{Def}
The \emph{Hausdorff distance} $\d_\hausd : \Conv_{\H} \times \Conv_{\H} \function \clint{0,\infty}$ is defined by
\begin{equation}\notag
  \d_{\hausd}(\A,\B) := 
  \max\left\{\sup_{a \in \A} \d(a,\B)\,,\, \sup_{b \in \B} \d(b,\A)\right\}, \qquad
  \A, \B \in \Conv_{\H}.
\end{equation}
\end{Def}

The metric space $(\Conv_\H, \d_\hausd)$ is complete (cf. \cite[Theorem II-14, Section II.3.14, p. 47]{CasVal77}).


\subsection{Differential measures}

We recall that a \emph{$\H$-valued measure on $I$} is a map $\mu : \lebsets(I) \function \H$ such that 
$\mu(\bigcup_{n=1}^{\infty} B_{n})$ $=$ $\sum_{n = 1}^{\infty} \mu(B_{n})$ whenever $(B_{n})$ is a sequence of mutually disjoint sets in $\lebsets(I)$. The \emph{total variation of $\mu$} is the positive measure 
$\vartot{\mu} : \lebsets(I) \function \clint{0,\infty}$ defined by
\begin{align}\label{tot var measure}
  \vartot{\mu}(B)
  := \sup\left\{\sum_{n = 1}^{\infty} \norm{\mu(B_{n})}{}\ :\ 
                 B = \bigcup_{n=1}^{\infty} B_{n},\ B_{n} \in \lebsets(I),\ 
                 B_{h} \cap B_{k} = \varnothing \text{ if } h \neq k\right\}. \notag
\end{align}
The vector measure $\mu$ is said to be \emph{with bounded variation} if $\vartot{\mu}(I) < \infty$. In this case the equality $\norm{\mu}{} := \vartot{\mu}(I)$ defines a norm on the space of measures with bounded variation (see, e.g. 
\cite[Chapter I, Section  3]{Din67}). 

If $\nu : \lebsets(I) \function \clint{0,\infty}$ is a positive bounded Borel measure and if $g \in \L^1(I,\nu;\H)$, then $g\nu$ will denote the vector measure defined by $g\nu(B) := \int_B g\de \nu$ for every $B \in \lebsets(I)$. In this case 
$\vartot{g\nu}(B) = \int_B \norm{g(t)}{}\de \nu$ for every $B \in \lebsets(I)$ (see \cite[Proposition 10, p. 174]{Din67}). Moreover a vector measure $\mu$ is called \emph{$\nu$-absolutely continuous} if $\mu(B) = 0$ whenever 
$B \in \lebsets(I)$ and $\nu(B) = 0$. 

Assume that $\mu : \lebsets(I) \function \H$ is a vector measure with bounded variation and let $f : I \function \H$ and 
$\phi : I \function \ar$ be two \emph{step maps with respect to $\mu$}, i.e. there exist $f_{1}, \ldots, f_{m} \in \H$, 
$\phi_{1}, \ldots, \phi_{m} \in \ar$ and $A_{1}, \ldots, A_{m} \in \lebsets(I)$ mutually disjoint such that 
$\vartot{\mu}(A_{j}) < \infty$ for every $j$ and $f = \sum_{j=1}^{m} \indicator_{A_{j}} f_{j},$, 
$\phi = \sum_{j=1}^{m} \indicator_{A_{j}} \phi_{j},$ where $\indicator_{S} $ is the characteristic function of a set $S$, i.e. 
$\indicator_{S}(x) := 1$ if $x \in S$ and $\indicator_{S}(x) := 0$ if $x \not\in S$. For such step functions we define 
$\int_{I} \duality{f}{\mu} := \sum_{j=1}^{m} \duality{f_{j}}{\mu(A_{j})} \in \ar$ and
$\int_{I} \phi \de \mu := \sum_{j=1}^{m} \phi_{j} \mu(A_{j}) \in \H$. If $\Step(\vartot{\mu};\H)$ (resp. $\Step(\vartot{\mu})$) is the set of $\H$-valued (resp. real valued) step maps with respect to $\mu$, then the maps
$\Step(\vartot{\mu};\H)$ $\function$ $\ar : f \longmapsto \int_{I} \duality{f}{\mu}$ and
$\Step(\vartot{\mu})$ $\function$ $\H : \phi \longmapsto \int_{I} \phi \de \mu$ 
are linear and continuous when $\Step(\vartot{\mu};\H)$ and $\Step(\vartot{\mu})$ are endowed with the 
$\L^{1}$-seminorms $\norm{f}{\L^{1}(\vartot{\mu};\H)} := \int_I \norm{f}{} \de \vartot{\mu}$ and
$\norm{\phi}{\L^{1}(\vartot{\mu})} := \int_I |\phi| \de \vartot{\mu}$. Therefore they admit unique continuous extensions 
$\mathsf{I}_{\mu} : \L^{1}(\vartot{\mu};\H) \function \ar$ and 
$\mathsf{J}_{\mu} : \L^{1}(\vartot{\mu}) \function \H$,
and we set 
\[
  \int_{I} \duality{f}{\de \mu} := \mathsf{I}_{\mu}(f), \quad
  \int_{I} \phi \mu := \mathsf{J}_{\mu}(\phi),
  \qquad f \in \L^{1}(\vartot{\mu};\H),\quad \phi \in \L^{1}(\vartot{\mu}).
\]

If $\nu$ is bounded positive measure and $g \in \L^{1}(\nu;\H)$, arguing first on step functions, and then taking limits, it is easy to check that 
\begin{equation}\label{f de gnu =fg de nu}
  \int_I\duality{f}{\de(g\nu)} = \int_I \duality{f}{g}\de \nu \qquad \forall f \in \L^{\infty}(\mu;\H).
\end{equation} 
The following results (cf., e.g., \cite[Section III.17.2-3, pp. 358-362]{Din67}) provide a connection between functions with bounded variation and vector measures which will be implicitly used in the paper.

\begin{Thm}\label{existence of Stietjes measure}
For every $f \in \BV(I;\H)$ there exists a unique vector measure of bounded variation $\mu_{f} : \lebsets(I) \function \H$ such that 
\begin{align}
  \mu_{f}(\opint{c,d}) = f(d-) - f(c+), \qquad \mu_{f}(\clint{c,d}) = f(d+) - f(c-), \notag \\ 
  \mu_{f}(\clsxint{c,d}) = f(d-) - f(c-), \qquad \mu_{f}(\cldxint{c,d}) = f(d+) - f(c+). \notag 
\end{align}
whenever $c < d$ and the left hand side of each equality makes sense. Conversely, if $\mu : \lebsets(I) \function \H$ is a vector measure with bounded variation, and if $f_{\mu} : I \function \H$ is defined by 
$f_{\mu}(t) := \mu(\clsxint{\inf I,t} \cap I)$, then $f_{\mu} \in \BV(I;\H)$ and $\mu_{f_{\mu}} = \mu$.
\end{Thm}

\begin{Prop}
Let $f  \in \BV(I;\H)$, let $g : I \function \H$ be defined by $g(t) := f(t-)$, for $t \in \Int(I)$, and by $g(t) := f(t)$, if 
$t \in \partial I$, and let $V_{g} : I \function \ar$ be defined by $V_{g}(t) := \pV(g, \clint{\inf I,t} \cap I)$. Then  
$\mu_{g} = \mu_{f}$ and $\vartot{\mu_{f}} = \mu_{V_{g}} = \pV(g,I)$.
\end{Prop}

The measure $\mu_{f}$ is called \emph{Lebesgue-Stieltjes measure} or \emph{differential measure} of $f$. Let us see the connection with the distributional derivative. If $f \in \BV(I;\H)$ and if $\overline{f}: \ar \function \H$ is defined by
\begin{equation}\label{extension to R}
  \overline{f}(t) :=
  \begin{cases}
    f(t) 	& \text{if $t \in I$} \\
    f(\inf I)	& \text{if $\inf I \in \ar$, $t \not\in I$, $t \le \inf I$} \\
    f(\sup I)	& \text{if $\sup I \in \ar$, $t \not\in I$, $t \ge \sup I$}
  \end{cases},
\end{equation}
then, as in the scalar case, it turns out (cf. \cite[Section 2]{Rec11a}) that $\mu_{f}(B) = \D \overline{f}(B)$ for every 
$B \in \lebsets(\ar)$, where $\D\overline{f}$ is the distributional derivative of $\overline{f}$, i.e.
\[
  - \int_\ar \varphi'(t) \overline{f}(t) \de t = \int_{\ar} \varphi \de \D \overline{f} 
  \qquad \forall \varphi \in \Czero_{c}^{1}(\ar;\ar),
\]
$\Czero_{c}^{1}(\ar;\ar)$ being the space of real continuously differentiable functions on $\ar$ with compact support.
Observe that $\D \overline{f}$ is concentrated on $I$: $\D \overline{f}(B) = \mu_f(B \cap I)$ for every $B \in \lebsets(I)$, hence in the remainder of the paper, if $f \in \BV(I,\H)$ then we will simply write
\begin{equation}
  \D f := \D\overline{f} = \mu_f, \qquad f \in \BV(I;\H),
\end{equation}
and from the previous discussion it follows that 
\begin{equation}
  \norm{\D f}{} = \vartot{\D f}(I) = \norm{\mu_f}{}  = \pV(f,I) \qquad \forall f \in \BV^\r(I;\H).
\end{equation}
If $I$ is bounded and $p \in \clint{1,\infty}$, then the classical Sobolev space $\W^{1,p}(I;\H)$ consists of those functions $f \in \Czero(I;\H)$ such that $\D f = g\leb^1$ for some $g \in \L^p(I;\H)$ and we have  $\W^{1,p}(I;\H) = \AC^p(I;\H)$. Let us also recall that if $f \in \W^{1,1}(I;\H)$ then {}{ the derivative $f'(t)$  exists} for $\leb^1$-a.e. in $t \in I$, 
$\D f = f' \leb^1$, and $\V(f,I) = \int_I\norm{f'(t)}{}\de t$ (cf., e.g. \cite[Appendix]{Bre73}).

In \cite[Lemma 6.4 and Theorem 6.1]{Rec11a} it is proved that

\begin{Prop}\label{P:BV chain rule}
Assume that $J \subseteq \ar$ is a bounded interval and $h : I \function J$ is nondecreasing.
\begin{itemize}
\item[(i)]
  $\D h(h^{-1}(B)) = \leb^{1}(B)$ for every $B \in \lebsets(h(\cont(h)))$.
\item[(ii)]
 If $f \in \Lip(J;\H)$ and $g : I \function \H$ is defined by
\begin{equation}\notag
  g(t) := 
  \begin{cases}
    f'(h(t)) & \text{if $t \in \cont(h)$} \\
    \ \\
    \dfrac{f(h(t+)) - f(h(t-))}{h(t+) - h(t-)} & \text{if $t \in \discont(h)$}
  \end{cases},
\end{equation}
then $f \circ h \in \BV(I;\H)$ and $\D\ \!(f \circ h) = g \D h$. This result holds even if $f'$ is replaced by any of its 
$\leb^{1}$-representatives.
\end{itemize}
\end{Prop}


\section{Statement of the main result}\label{S:state main result}

In this section we state the main theorems of the present paper. We assume that

\begin{gather}
   \Z \in \Conv_\H, \qquad 0 \in \Z, \label{Z closed convex-results}  \\
   T \in \opint{0,\infty}. \label{T>0-results}
\end{gather}

\begin{Thm}\label{main thm}
Assume that $u \in \BV(\clint{0,T};\H)$ and $z_0 \in \Z$. Then there exists a unique $y \in \BV(\clint{0,T};\H)$ such that  
\begin{alignat}{3}
  & u(t) - y(t) \in \Z & \qquad & \forall t \in \clint{0,T}, \label{BV constraint for P} \\
  & \int_{\cont(u)} \duality{z(t) - u(t) + y(t)}{\de\D y(t)} \le 0 & \quad & 
                       \text{$\forall z \in \L^\infty(\clint{0,T};\H)$, $z(\clint{0,T}) \subseteq \Z$}, 
      \label{BV int ineq for P} \\
  & u(t) - y(t) = \Proj_{\Z}(u(t) - y(t-)) & \qquad &  \forall t \in \discont(\C) \setmeno \{0\}, \label{BV jump cond for P-} \\
  &  u(t+) - y(t+) = \Proj_{\Z}(u(t+) - y(t)) & \qquad &  \forall t \in \discont(\C), 
       \label{BV jump cond for P+} \\
  & u(0) - y(0) = z_0. \label{BV i.c. for P}
\end{alignat}
Moreover $y$ is left continuous (respectively: right continuous) at $t \in \clint{0,T}$ if and only if $u$ is left continuous (respectively: right continuous) at $t \in \clint{0,T}$.
\end{Thm}

The following result shows that in the left continuous case the three conditions 
\eqref{BV int ineq for P}--\eqref{BV jump cond for P+} reduce to the single one
\begin{equation}\label{single integral condition}
\int_{\clint{0,T}} \duality{z(t) - u(t+) + y(t+)}{\de\D y(t)} \le 0 \quad  
                       \text{$\forall z \in \L^\infty(\clint{0,T};\H)$, $z(\clint{0,T}) \subseteq \Z$}.
\end{equation}

\begin{Thm}\label{main-left continuous case}
Assume that $u \in \BV^\l(\clint{0,T};\H)$ and $z_0 \in \Z$, and let $y \in \BV(\clint{0,T};\H)$ be the unique function satisfying \eqref{BV constraint for P}--\eqref{BV i.c. for P}. Then $y$ is left continuous and we have
\begin{alignat}{3}
  & u(t) - y(t) \in \Z & \qquad & \forall t \in \clint{0,T}, \label{BV-l constraint for P} \\
  & \int_{\clint{0,T}} \duality{z(t) - u(t+) + y(t+)}{\de\D y(t)} \le 0 & \quad & 
                       \text{$\forall z \in \L^\infty(\clint{0,T};\H)$, $z(\clint{0,T}) \subseteq \Z$}, 
      \label{BV-l int ineq for P} \\
  & u(0) - y(0) = z_0. \label{BV-l i.c. for P}
\end{alignat}
\end{Thm}

As a consequence we obtain that in the left continuous case our general formulation is equivalent to the well known formulation of the play operator provided in \cite{KreLau02}, namely \eqref{BV-l constraint for P} and \eqref{BV-l i.c. for P} together with
\begin{equation}\label{KreLau integral formulation}
\int_{\clint{0,T}} \duality{z(t) - u(t+) + y(t+)}{\de y(t)} \le 0 \quad  
                       \text{$\forall z \in \L^\infty(\clint{0,T};\H)$, $z(\clint{0,T}) \subseteq \Z$},
\end{equation}
where the integral in \eqref{KreLau integral formulation} is meant in the sense of Young (cf. \cite[Section 3]{KreLau02}).
Indeed in \cite[Section A4]{Rec11a} we showed that the two integrals in \eqref{single integral condition} and in \eqref{KreLau integral formulation} coincide.

An analogous result holds in the right continuous case:

\begin{Thm}\label{main-right continuous case}
Assume that $u \in \BV^\r(\clint{0,T};\H)$ and $z_0 \in \Z$, and let $y \in \BV(\clint{0,T};\H)$ be the unique function satisfying \eqref{BV constraint for P}--\eqref{BV i.c. for P}. Then $y$ is right continuous and we have
\begin{alignat}{3}
  & u(t) - y(t) \in \Z & \qquad & \forall t \in \clint{0,T}, \label{BV-r constraint for P} \\
  & \int_{\clint{0,T}} \duality{z(t) - u(t) + y(t)}{\de\D y(t)} \le 0 & \quad & 
                       \text{$\forall z \in \L^\infty(\clint{0,T};\H)$, $z(\clint{0,T}) \subseteq \Z$}, 
      \label{BV-r int ineq for P} \\
  & u(0) - y(0) = z_0. \label{BV-r i.c. for P}
\end{alignat}
\end{Thm}


\section{Proofs}\label{S:Proofs}

We start by recalling the following result about sweeping processes.

\begin{Thm}\label{T:existence general Lipsweep}
If $\C \in \Lip(\clint{0,T};\Conv_\H)$ and $y_0 \in \H$, then there is a unique $y \in \Lip(\clint{0,T};\H)$ such that 
\begin{alignat}{3}
  & y(t) \in \C(t) & \qquad & \forall t \in \clint{0,T}, \label{y in C - Lip-sweep} \\
  & y'(t) + N_{\C(t)}(y(t)) \ni 0 & \qquad & \text{for $\leb^1$-a.e. $t \in \clint{0,T}$}, \label{diff. incl. - Lipsweep} \\
  & y(0) = \Proj_{\C(0)}(y_{0}). \label{in. cond. - Lipsweep}
\end{alignat}
If we set $\M(y_0,\C) := y$ then we have defined a solution operator 
$\M : \Lip(\clint{0,T};\Conv_\H) \times \H$ $\function$ $\Lip(\clint{0,T};\H)$ assigning to $(\C,y_0)$ the unique function
$y$ satisfying \eqref{y in C - Lip-sweep}--\eqref{in. cond. - Lipsweep}. This operator satisfies the semigroup property
\begin{equation}\label{semigroup}
  \M(\C,y_0)(t) = \M(\C(\cdot + s), \M(\C,y_0)(s))(t-s) \qquad \forall t, s, \ 0 \le s \le t,
\end{equation}
where $\C(\cdot + s) : \clint{0,T-s} \function \Conv_\H$ is defined by $\C(\cdot + s)(t) := \C(t + s)$, for $t \in \clint{0,T-s}$.
\end{Thm}

The proof of the previous theorem can be found in \cite[Section 3]{Mor71} (see also \cite[Proposition 3c]{Mor77} for a more general setting).

In the following proposition we recall a family of curves with values in $\Conv_\H$ which play a special role 
in sweeping processes (cf. \cite[Proposition 4.1]{Rec16a} and \cite[Theorem 1]{Ser98}).

\begin{Prop}
If we set
\begin{equation}
  D_\rho := \{u \in \H \ :\ \norm{u}{} \le \rho\}, \qquad \rho > 0,
\end{equation}
then for every pair $\A,\B \in \Conv_\H$ with $\A \neq \B$ and $\rho := \d_\hausd(\A,\B) < \infty$, we define the function
$\G_{\A,\B} : \clint{0,1} \function \Conv_\H$ by setting
\begin{equation}\label{catching-up geodesic-2}
  \G_{(\A,\B)}(t) := (\A + D_{t\rho}) \cap (\B + D_{(1-t)\rho}), \quad t \in \clint{0,1}.
\end{equation}
Then $\G_{(\A,\B)} \in \Lip(\clint{0,1};\Conv_\H)$ and it is a geodesic connecting $\A$ and $\B$ in $\Conv_\H$, i.e. 
$\G_{(\A,\B)}(0) = \A$, $\G_{(\A,\B)}(1) = \B$ and $\pV(\G_{(\A,\B)}, \clint{0,1}) = \d_\hausd(\A,\B)$.
\end{Prop}

As stated in the previous proposition the curve $\G_{(\A,\B)}$ represents a sort of minimal path from $\A$ to $\B$ in 
$\Conv_{\H}$ (but this is not the only one, as shown in \cite{Ser98} and in \cite{Rec11c}). We will see that a connection with the play operator is obtained when $\G_{(\A,\B)}$ joins two sets of the form $\A = u_0 - \Z$ and $\B = u_1 - \Z$. 

Before stating the next Lemma, let us observe that $\d_\hausd(u_0 - \Z, u_1 - \Z) = \norm{u_0-u_1}{}$ for every pair
$u_0, u_1 \in \H$.

\begin{Lem}\label{particular sweep}
Assume that $u_0, u_1 \in \H$ with $u_0 \neq u_1$,  and let $\G_{u_0, u_1} : \clint{0,1} \function \Conv_\H$ be defined by $\G_{u_0, u_1} := \G_{(u_0 - \Z, u_1 - \Z)}$, i.e. 
\begin{equation}\label{convex-valued geodesic}
  \G_{u_0,u_1}(t) := (u_0 - \Z + D_{t\norm{u_0-u_1}{}}) \cap (u_1 - \Z + D_{(1-t)\norm{u_0-u_1}{}}), \qquad t \in \clint{0,1}.
\end{equation}
If $y_0 \in u_0 - \Z$ then there exists a unique $y \in \Lip(\clint{0,1};\H)$ such that
\begin{align}
  & y(t) \in \G_{u_0,u_1}(t) \qquad \forall t \in \clint{0,1}, \label{Pjump sweep-constr} \\
  & y'(t) \in -N_{\G_{u_0,u_1}(t)}(y(t))   \qquad \text{for $\leb^1$-a.e. $t \in \clint{0,1}$}, \label{Pjump sweep-eq} \\
  & y(0) = y_0. \label{Pjump sweep-i.c.}
\end{align}
Moreover if $t_0 \in \clint{0,1}$ is the unique number such that 
\begin{equation}
  \d(y_0, u_1 - \Z) = (1 - t_0)\norm{u_1 - u_0}{},
\end{equation}
then one has
\begin{equation}\label{sol. of the particular sweeping pb}
  y(t) = 
  \begin{cases}
    y_0 & \text{if $t \in \clsxint{0,t_0}$}, \\
    y_0 + \dfrac{t-t_0}{1-t_0}(\Proj_{u_1-\Z}(y_0) - y_0) & \text{if $t_0 \neq 1$, $t \in \clsxint{t_0,1}$}, \\
    \Proj_{u_1-\Z}(y_0) & \text{if $t = 1$}.
  \end{cases}
\end{equation}
\end{Lem}

\begin{proof}
The Lipschitz continuity of $\G_{u_0,u_1}$ can be inferred from \cite[Proposition 4.4]{Rec16a}, therefore the existence and uniqueness of a function $y$ solving \eqref{Pjump sweep-constr}--\eqref{Pjump sweep-i.c.} is a consequence of the Theorem \ref{T:existence general Lipsweep} about sweeping processes. The fact that $y$ is explicitly given by 
\eqref{sol. of the particular sweeping pb} is proved in \cite[Lemma 4.5]{Rec16a}.
\end{proof}

\begin{Rem}
Observe, in the above Lemma, that if $y_0 \not\in u_1 - \Z$ then $t_0$ is the first time $t$ when the boundary of
$u_1 - \Z + D_{(1-t)\norm{u_1-u_0}{}}$ touches $y_0$. From that moment on, $y(t)$ starts moving along the segment with endpoints $y_0$ and $y_1$ and it finishes its evolution on the point $y_1 = y(1)$: in other words at the time $t_0$ the point $y_0$ ``is swept'' by the moving boundary of $u_1 - \Z + D_{(1-t)\norm{u_1-u_0}{}}$ along a segment ending at $y_1$.
\end{Rem}

The point in Lemma \ref{particular sweep} is that the function $\G_{u_0,u_1}$ is such that the solution $y$ of 
\eqref{Pjump sweep-constr}--\eqref{Pjump sweep-i.c.} is such that $y(1) = \Proj_{u_1-\Z}(y_0)$ for every
initial condition $y_0 \in \G_{u_0,u_1}(0) = u_0 - \Z$.

In the next proposition we introduce a general technique of reparametrization by the arc length for functions of bounded 
variation with values in a metric space. This technique is a slight generalization of \cite[Proposition 5.1]{Rec16a}.

\begin{Prop}\label{D:reparametrization}
Assume that \eqref{X complete metric space} is satisfied, set
\begin{equation}\label{pairs of X with finite distance}
  \Phi_X := \{(x,y) \in \X \times \X\ :\ 0 < d(x,y) < \infty\}, \notag 
\end{equation}
and let $\mathscr{G} = (g_{(x,y)})_{(x,y) \in \Phi}$ be a family of geodesics connecting $x$ to $y$ for every 
$(x,y) \in \Phi_X$.
For every $f \in \BV(\clint{a,b};\X)$ we define $\ell_f : \clint{0,T} \function \clint{0,T}$ by
\begin{equation}\label{normalized arc length}
  \ell_f(t) := 
  \begin{cases}
    \dfrac{T}{\pV(f,\clint{0,T})}\pV(f,\clint{0,t}) & \text{if $\pV(f,\clint{0,T}) \neq 0$}, 	\\
    															\\
    0							           & \text{if $\pV(f,\clint{0,T}) = 0$},
  \end{cases}
    \qquad t \in \clint{0,T}.
\end{equation}
Then then there is a unique $\ftilde \in \Lip(\clint{0,T};\X)$ such that $\Lipcost(\ftilde) \le \pV(f,\clint{0,T})/T$ and 
\begin{align}
  & f(t) = \ftilde(\ell_f(t)) \qquad \forall t \in \clint{0,T}, \label{reparametrization1}  
  \\ 
 & \ftilde(\sigma) = g_{(f(t-),f(t))}\left(\frac{\sigma - \ell_f(t-)}{\ell_f(t) - \ell_f(t-)}\right) 
          \qquad \forall \sigma \in \clint{\ell_f(t-),\ell_f(t)} \quad \text{if $\ell_f(t-) \neq \ell_f(t)$}, \label{reparametrization2} \\ 
  & \ftilde(\sigma) = g_{(f(t),f(t+))}\left(\frac{\sigma - \ell_f(t)}{\ell_f(t+) - \ell_f(t)}\right) 
          \qquad \forall \sigma \in \clint{\ell_f(t),\ell_f(t+)}\quad \text{if $\ell_f(t) \neq \ell_f(t+)$}.  \label{reparametrization3}
\end{align}
\end{Prop}

\begin{proof}
The existence and uniqueness of a Lipschitz continuous function $F : \ell_f(\clint{0,T}) \function \X$ such that
$f(t) = F(\ell_f(t))$ for every $t$ is obtained in a standard way (see, e.g., \cite[Section 2.5.16, p. 109]{Fed69}). The uniqueness of an extension of $\ftilde : \clint{0,T}\function \X$ of $F$ satisfying 
\eqref{reparametrization2}--\eqref{reparametrization3} is a straighforward consequence of the fact that $g_{(f(t-),f(t))}$ and $g_{(f(t),f(t+))}$ are two geodesics connecting respectively $f(t-) = F(\ell_f(t-))$ with $f(t)= F(\ell_f(t))$, and 
$f(t)= F(\ell_f(t))$ with $f(t+)= F(\ell_f(t+))$.
\end{proof}

The technique introduced in the previous Proposition differs from the classical one in \cite[Section 2.5.16, p. 109]{Fed69}, since the jumps of the given function are connected by geodesics curves which values in the original metric space $\X$, and not in a Banach space where $\X$ is embedded.

If in Proposition \ref{D:reparametrization} we take $\X = \H$ and $\mathscr{G} = (g_{(x,y)})_{(x,y) \in \Phi_\H}$ defined by
\[
  g_{(x,y)}(t) := (1-t)x+ty, \qquad t \in \clint{0,1},
\]
we obtain the following

\begin{Cor}\label{reparametrization in H}
If $u \in \Lip(\clint{0,T};\H)$ then there exists a unique $\utilde \in \Lip(\clint{0,T};\H)$ such that 
$\Lipcost(\utilde,\clint{0,T}) \le \V(u,\clint{0,T})/T$ and 
\begin{align}
  & u(t) = \utilde(\ell_u(t))   \qquad \forall t \in \clint{0,T}, \label{H-reparametrization1} \\ 
  & \utilde(\ell_u(t-)(1-\lambda) + \ell_u(t)\lambda) =  
    (1-\lambda)u(t-) + \lambda u(t) 
     \qquad  \forall \lambda \in \clint{0,1} \quad \text{if $\ell_u(t-) \neq \ell_u(t)$}, \label{H-reparametrization2} \\
  & \utilde(\ell_u(t)(1-\lambda) + \ell_u(t+)\lambda) =  
    (1-\lambda)u(t) + \lambda u(t+) 
     \qquad \forall \lambda \in \clint{0,1} \quad \text{if $\ell_u(t) \neq \ell_u(t+)$}. \label{H-reparametrization3}
\end{align}
\end{Cor}

If instead we choose $X = \Conv_\H$ and $\mathscr{G} = (\G_{(\A,\B)})_{(\A,\B) \in \Phi_{\Conv_\H}}$ provided by
\[
  \G_{(\A,\B)} = (\A + D_{t\rho}) \cap (\B + D_{(1-t)\rho}), \quad t \in \clint{0,1}, \qquad \rho = \d_\hausd(\A,\B),
\]
then the following Corollary is inferred.

\begin{Cor}\label{reparametrization in Conv}
If $\C \in \Lip(\clint{0,T};\Conv_\H)$ then there exists a unique $\Ctilde \in \Lip(\clint{0,T};\Conv_\H)$ such that 
$\Lipcost(\Ctilde,\clint{0,T}) \le \V(\C,\clint{0,T})/T$ and 
\begin{align}
  & \C(t) = \Ctilde(\ell_\C(t))   \qquad \forall t \in \clint{0,T}, \label{C-reparametrization1} \\ 
  & \Ctilde(\ell_\C(t-)(1-\lambda) + \ell_\C(t)\lambda)  =  
     (\C(t-) + D_{\lambda \rho_{t-}}) \cap (\C(t) + D_{(1-\lambda)\rho_{t-}}) \notag \\
  & \phantom{\ \ \ \ \ \ \ \ \ \ \ \ \ \ \ \ }
     \qquad \forall \lambda \in \clint{0,1}, \text{if $\ell_u(t-) \neq \ell_u(t)$ with $\rho_{t-} := \d_\hausd(\C(t-),\C(t))$}. \\
   & \Ctilde(\ell_\C(t)(1-\lambda) + \ell_\C(t+)\lambda)  =  
     (\C(t) + D_{\lambda \rho_{t+}}) \cap (\C(t+) + D_{(1-\lambda)\rho_{t+}}) \notag \\
  & \phantom{\ \ \ \ \ \ \ \ \ \ \ \ \ \ \ \ }
     \qquad \forall \lambda \in \clint{0,1}, \text{if $\ell_u(t) \neq \ell_u(t+)$ with $\rho_{t+} := \d_\hausd(\C(t),\C(t+))$}.   
\end{align}
\end{Cor}

If $u \in \BV(\clint{0,T};\H)$ and we take $\C = \C_u := u - \Z \in \BV(\clint{0,T};\Conv_\H)$, then there is a relationship between the two situations in Corollary \ref{reparametrization in H} and Corollary \ref{reparametrization in Conv}:

\begin{Lem}\label{Ctilde_u = utilde - Z}
Assume that $u \in \BV(\clint{0,T};\H)$ and let $\C_u : \clint{0,T} \function \Conv_\H$ be defined by 
\begin{equation}
  \C_u(t) := u(t) - \Z, \qquad t \in \clint{0,T}.
\end{equation}
Then $\C_u \in \BV(\clint{0,T};\Conv_\H)$, $\ell_u = \ell_{\C_u}$ and 
\begin{equation}\label{Ctilde_u = utilde - Z eq}
  \Ctilde_u(\sigma) = \utilde(\sigma) - \Z  \qquad \text{on $\ell_u(\clint{0,T})$},
\end{equation} 
where $\ell_u$ and $\ell_{\C_u}$ are the normalized arc lengths of $u$ and $\C_u$, and $\utilde$ and $\Ctilde_u$ are the reparametrizations defined respectively in Corollaries \ref{reparametrization in H} and \ref{reparametrization in Conv}.
\end{Lem}  

\begin{proof}
First of all let us observe that for every $t, s \in \clint{0,T}$ we have that 
$\d_\hausd(\C_u(t),\C_u(s)) = \d_\hausd(u(t)-\Z,u(s)-\Z) = \norm{u(t) - u(s)}{}$, therefore $\pV(\C_u,J) = \pV(u,J)$ for every interval $J \subseteq \clint{0,T}$, and this implies that $\C_u \in \BV(\clint{0,T};\Conv_\H)$ and $\ell_{\C_u} = \ell_u$.
Hence, for every $\sigma \in \ell_u(\clint{0,T})$ there exists $t \in \clint{0,T}$ such that $\sigma = \ell_u(t) = \ell_{\C_u}(t)$ and we have that 
\begin{align}
  \utilde(\sigma) - \Z 
    = \utilde(\ell_u(t)) - \Z 
     = u(t) - \Z  = \C_u(t)
     = \Ctilde_u(\ell_\C(t)) = \Ctilde_u(\sigma). \notag
\end{align}
\end{proof}

Now we are finally in position to provide the

\begin{proof}[Proof of Theorem \ref{main thm}]
Let $\C_u : \clint{0,T} \function \Conv_\H$ be defined by
\begin{equation}
  \C_u(t) := u(t) - \Z, \qquad t \in \clint{0,T}.
\end{equation}
Thank to Lemma \ref{Ctilde_u = utilde - Z} we have that $\C_u \in \BV(\clint{0,T};\Conv_\H)$, and if 
$\ell_u : \clint{0,T} \function \clint{0,T}$ is the normalized arc length of $u$ defined by \eqref{normalized arc length}, 
then by Proposition \ref{D:reparametrization} there exists a unique $\utilde \in \Lip(\ell_u(\clint{0,T});\H)$ such that
\begin{equation}
  u(t) = \utilde(\ell_u(t)) \qquad \forall t \in \clint{0,T}
\end{equation}
(actually we will not need to extend $\utilde$ outside of $\ell_u(\clint{0,T})$).
Now let $\ell_{\C_u}  : \clint{0,T} \function \clint{0,T}$ be the normalized arc length of $\C_u$, let 
$\Phi_{\Conv_\H} = \{(\A,\B) \in \Conv_\H \times \Conv_\H \ :\ 0 < \d_\hausd(\A,\B) < \infty\}$, and for every
$(\A,\B) \in \Phi_{\Conv_\H}$ let $\mathscr{G}_{(\A,\B)} : \clint{0,T} \function \Conv_\H$ be defined by 
\eqref{catching-up geodesic-2}. Then by Corollary \ref{reparametrization in Conv} there exists a unique reparametrization
$\Ctilde_u \in \Lip(\clint{0,T};\Conv_\H)$ such that $\Lipcost(\Ctilde_u) \le \V(\C_u,\clint{0,T})/T$ and
\begin{align}
  & u(t) - \Z = \Ctilde_u(\ell_{\C_u}(t)) = \utilde(\ell_u(t)) - \Z \qquad \forall t \in \clint{0,T}, \label{u - Z = Ctilde_u(lu)}\\
  &  \Ctilde_u(\sigma) = \G_{u(t-),u(t)}\left(\dfrac{\sigma - \ell_u(t-)}{\ell_u(t) - \ell_u(t-)}\right) 
         \qquad \forall \sigma \in \clint{\ell_u(t-),\ell_u(t)} \quad \text{if $\ell_u(t-) \neq \ell_u(t)$}, \label{Ctilde - } \\ 
  & \Ctilde_u(\sigma) = \G_{u(t),u(t+)}\left(\dfrac{\sigma - \ell_u(t)}{\ell_u(t+) - \ell_u(t)}\right) 
          \qquad \forall \sigma \in \clint{\ell_u(t),\ell_u(t+)}\quad \text{if $\ell_u(t) \neq \ell_u(t+)$}, \label{Ctilde + }
\end{align}
where $\G_{u(t-),u(t)}$ and $\G_{u(t),u(t+)}$ are given by \eqref{convex-valued geodesic} and we have used 
Lemma \ref{Ctilde_u = utilde - Z} in \eqref{u - Z = Ctilde_u(lu)}.
Since $\Ctilde_u \in \Lip(\clint{0,T};\Conv_\H)$ we can define
\begin{equation}\label{M(Ctilde_u)}
  \yhat := \M(\Ctilde_u, u(0)-z_0),
\end{equation}
where $\M$ is the solution operator of the sweeping process defined in Theorem \ref{T:existence general Lipsweep}, and  we define the function $y : \clint{0,T} \function \H$ by setting 
\begin{equation}
  y(t) := \M(\Ctilde_u, u(0)-z_0)(\ell_u(t)) = \yhat(\ell_u(t)), \qquad t \in \clint{0,T}.
\end{equation}   
We claim that $y$ is of bounded variation and solves problem \eqref{BV constraint for P}--\eqref{BV i.c. for P}. First of all let us observe that thanks to \eqref{y in C - Lip-sweep} and \eqref{Ctilde_u = utilde - Z eq} we have that  for every 
$t \in \clint{0,T}$ 
\begin{equation}
  y(t) = \M(\Ctilde_u, u(0)-z_0)(\ell_u(t)) \in \Ctilde_u(\ell_u(t)) = \utilde(\ell_u(t)) - \Z = u(t) - \Z,
\end{equation} 
therefore $u(t) - y(t) \in \Z$ for every $t \in \clint{0,T}$ and \eqref{BV constraint for P} is satisfied.
Since 
$\yhat = \M(\Ctilde_u, u(0)-z_0)$ is Lipschitz continuous and $\ell_u$ is increasing, it is clear that $y \in \BV(\clint{0,T};\H)$ and that $y$ is left continuous (respectively: right continuous) if and only if $\ell_u$ is left continuous (respectively: right continuous), so that $\discont(y) = \discont(\ell_u) = \discont(u)$. Now let $w : \clint{0,T} \function \H$ be defined by
\begin{equation}\label{def of w}
  w(t) :=
  \begin{cases}
    \yhat'(\ell_u(t)) & \text{if $t \in \cont(u)$}, \\
    \ \\
    \dfrac{\yhat(\ell_u(t+)) - \yhat(\ell_u(t-))}{\ell_u(t+)-\ell_u(t-)} & \text{if $t \in \discont(u)$}.
  \end{cases}
\end{equation}
Thanks to the chain rule in Proposition \ref{P:BV chain rule}-(ii) we have that 
\begin{equation}\label{Dy = wDl}
  \D y = \D\ \!(\yhat \circ \ell_u) = w \D \ell_u.
\end{equation} 
Moreover from \eqref{M(Ctilde_u)} and Theorem \ref{T:existence general Lipsweep} it follows that 
$-\yhat'(\sigma) \in N_{\Ctilde_u(\sigma)}(\yhat(\sigma))$ for $\leb^1$-a.e. $\sigma \in \clint{0,T}$, i.e.
\begin{equation}
  \duality{\yhat(\sigma) - v}{\yhat'(\sigma)} \le 0 \quad \forall v \in \Ctilde_u(\sigma),
   \qquad \text{for $\leb^1$-a.e. $\sigma \in \clint{0,T}$},
\end{equation}
and this means that if
\begin{equation}
  E_u := \{\sigma \in \clint{0,T}\ :\ 
                 \duality{\yhat(\sigma) - v_\sigma}{\yhat'(\sigma)} > 0\ \text{for some $v_\sigma \in \Ctilde_u(\sigma)$}\},
\end{equation}
then $E_u$ is Lebesgue measurable and 
\begin{equation}
  \leb^1(E_u) = 0.
\end{equation}
Therefore if 
\[
  F_u := \{\sigma \in \ell_u(\clint{0,T})\ :\ 
                 \duality{\yhat(\sigma) - \utilde(\sigma) + \zeta_\sigma}{\yhat'(\sigma)} > 0\ \text{for some $\zeta_\sigma \in \Z$}\} 
\]
then thanks to Lemma \ref{Ctilde_u = utilde - Z} we have that
\begin{align}
  F_u \subseteq \{\sigma \in \ell_u(\clint{0,T})\ :\ 
                 \duality{\yhat(\sigma) - v_\sigma}{\yhat'(\sigma)} > 0\ \text{for some $v_\sigma \in \Ctilde_u(\sigma)$}\} \subseteq E_u
\end{align}
so that $F_u$ is also Lebesgue measurable and
\begin{equation}\label{leb(F)=0}
  \leb^1(F_u) = 0.
\end{equation}
Fix a bounded measurable $z : \clint{0,T} \function \H$ such that $z(t) \in \Z$ for every $t \in \clint{0,T}$. 
From \eqref{def of w} and \eqref{M(Ctilde_u)} we infer that 
\begin{align}
  & \{t \in \cont(u)\ :\ \duality{y(t) - u(t) + z(t)}{w(t)} > 0\} \notag \\
  & = \{t \in \cont(u)\ :\ \duality{\yhat(\ell_u(t)) - \utilde(\ell_u(t)) + z(t)}{\yhat'(\ell_u(t))} > 0\} \notag \\
  & \subseteq \{t \in \cont(u)\ :\ \duality{\yhat(\ell_u(t)) - \utilde(\ell_u(t)) + \zeta_t}{\yhat'(\ell_u(t))} > 0 \ 
                                   \text{for some $\zeta_t \in \Z$}\} \notag \\
  & \subseteq \{t \in \cont(u)\ :\ \ell_u(t) \in F_u\} \notag \\
  & = \ell_u^{-1}(F_u) \cap \cont(u) \notag \\
  & \subseteq \ell_u^{-1}(F_u) \cap \ell_u^{-1}(\ell_u(\cont(u))) \notag \\
  & = \ell_u^{-1}(F_u \cap \ell_u(\cont(u))) \notag \\
  & \subseteq \ell_u^{-1}(F_u), \notag
\end{align}
thus from Proposition \ref{P:BV chain rule}-(i) and \eqref{leb(F)=0} it follows that
\begin{align}
   & \D\ell_u(\{t \in \cont(u)\ :\ \duality{y(t) - u(t) + z(t)}{w(t)} > 0\})  \notag \\
  & \le \D\ell_u(\ell_u^{-1}(F_u )) = \leb^1(F_u) = 0
\end{align}
therefore, from \eqref{Dy = wDl} and \eqref{f de gnu =fg de nu}, we get
\begin{align}
  & \int_{\cont(u)} \duality{y(t) - u(t) - z(t)}{\de \D y(t)} \notag \\
  & = \int_{\cont(u)} \duality{y(t) - u(t) - z(t)}{\de w(\D\ell_u)(t)} \notag \\
  & = \int_{\cont(u)} \duality{y(t) - u(t) - z(t)}{w(t)} \de\D\ell_u(t) \le 0, \notag
\end{align}
and \eqref{BV int ineq for P} is proved.
Now take $t \in \discont(u)$. Using \eqref{semigroup}, \eqref{Ctilde - }, and \eqref{sol. of the particular sweeping pb} we get
\begin{align}
  y(t) & = \yhat(\ell_u(t)) \notag \\
        & = \M(\Ctilde_u, u(0) - z_0)(\ell_u(t)) \notag \\
        & = \M\left(\Ctilde_u(\cdot + \ell_u(t-)), \M(\Ctilde_u,u(0)-z_0)(\ell_u(t-))\right)(\ell_u(t)-\ell_u(t-)) \notag \\
        & = \M\left(\G_{u(t-),u(t)}, \yhat(\ell_u(t-))\right)(1) \notag \\
        & = \M\left(\G_{u(t-),u(t)}, y(t-)\right)(1) \notag \\
        & = \Proj_{u(t) - \Z}(y(t-)), \notag
\end{align}
thus we have checked that $y(t) = \Proj_{u(t) - \Z}(y(t-))$ which is equivalent to \eqref{BV jump cond for P-}.
Similarly from \eqref{semigroup}, \eqref{Ctilde + }, and \eqref{sol. of the particular sweeping pb} we infer that
\begin{align}
  y(t+) & = \yhat(\ell_u(t+)) \notag \\
        & = \M(\Ctilde_u, u(0) - z_0)(\ell_u(t+)) \notag \\
        & = \M\left(\Ctilde_u(\cdot + \ell_u(t)), \M(\Ctilde_u,u(0)-z_0)(\ell_u(t))\right)(\ell_u(t+)-\ell_u(t)) \notag \\
        & = \M\left(\G_{u(t),u(t+)}, \yhat(\ell_u(t))\right)(1) \notag \\
        & = \M\left(\G_{u(t),u(t+)}, y(t)\right)(1) \notag \\
        & = \Proj_{u(t+) - \Z}(y(t)), \notag 
\end{align}
thus we have checked that $y(t+) = \Proj_{u(t+) - \Z}(y(t))$ which is equivalent to \eqref{BV jump cond for P+}.
Finally $y(0) = \M(\Ctilde_u,u(0)-z_0)(0) = \Proj_{u(0) - \Z}(u(0)-z_0) = u(0) - z_0$, hence the initial condition 
\eqref{BV i.c. for P} is satisfied. Now we have to prove that $y$ is the unique solution of 
\eqref{BV constraint for P}--\eqref{BV i.c. for P}. Assume by contradiction that there are two solutions $y_1$ and $y_2$.
Then If $B \in \mathscr{M}(\clint{0,T})$ then, by \cite[Proposition 2]{Mor76} and by taking $z = u - (y_1+y_2)/2$ in 
\eqref{BV int ineq for P} we get
\begin{align}
  \int_{B \cap \cont(u)} \de \D\!\ (\norm{y_1(\cdot) - y_2(\cdot)}{}^2) 
    & \le 2 \int_{B \cap \cont(u)} \duality{y_1 - y_2}{\de \D\!\ (y_1 - y_2)}  \notag \\
    & \le 2 \int_{\cont(u)} \duality{y_1 - y_2}{\de \D\!\ (y_1 - y_2)} \le 0, \notag
\end{align}
while if $t \in \discont(u)$, from \eqref{BV jump cond for P+}--\eqref{BV jump cond for P+} we infer that
\begin{align}
  & \D\!\ (\norm{y_1(\cdot) - y_2(\cdot)}{}^2)(\{t\}) \notag \\
    & =   \norm{y_1(t+) - y_2(t+)}{}^2 - \norm{y_1(t-) - y_2(t-)}{}^2 \notag \\
    & =   \norm{\Proj_{\Z}(u(t+) - y_1(t)) - \Proj_{\Z}(u(t+) - y_2(t))}{}^{{2}} - \norm{y_1(t-) - y_2(t-)}{}^2 \notag \\ 
    & \le \norm{y_1(t) - y_2(t)}{}^{{2}} - \norm{y_1(t-) - y_2(t-)}{}^2 \notag \\
    & = \norm{\Proj_{\Z}(u(t) - y_1(t-)) - \Proj_{\Z}(u(t) - y_2(t-))}{}^{{2}} - \norm{y_1(t-) - y_2(t-)}{}^2 \notag \\
    & \le \norm{y_1(t-) - y_2(t-)}{}^{{2}} - \norm{y_1(t-) - y_2(t-)}{}^2 = 0. \notag
\end{align}
Therefore for every $B \in \mathscr{M}(\clint{0,T})$ we find
\begin{align}
  & \D\!\ (\norm{y_1(\cdot) - y_2(\cdot)}{}^2)(B) \notag \\
    & = \D\!\ (\norm{y_1(\cdot) - y_2(\cdot)}{}^2)(B \cap \cont(u)) + 
            \D\!\ (\norm{y_1(\cdot) - y_2(\cdot)}{}^2)(B \cap \discont(u)) \notag \\
    & = \int_{B \cap \cont(u)} \de \D\!\ (\norm{y_1(\cdot) - y_2(\cdot)}{}^2) + 
          \sum_{t \in B \cap \discont(u)} \D\!\ (\norm{y_1(\cdot) - y_2(\cdot)}{}^2)(\{t\}) \le 0 \notag
\end{align}
which implies that $t \longmapsto \norm{y_1(t) - y_2(t)}{}^2$ is nonincreasing and leads to the uniqueness of the solution. 
\end{proof}

Now we provide the 

\begin{proof}[Proof of Theorem \ref{main-left continuous case}]
We know from Theorem \ref{main thm} that $y$ is left continuous, thus we only have to prove formula 
\eqref{BV-l int ineq for P}. If $t \in \discont(u)$ then \eqref{BV jump cond for P+} reads
\begin{equation}
  \duality{z - u(t+) + y(t+)}{y(t+) - y(t)} \le 0 \qquad \forall \zeta \in \Z,
\end{equation}
hence, since $\cont(y) = \cont(u)$, for every $z \in \L^\infty(\clint{0,T};\H)$ with $z(\clint{0,T}) \subseteq \Z$ we have
\begin{align}
  & \int_{\clint{0,T}} \duality{z(t) - u(t+) + y(t+)}{\de\D y(t)} \notag \\
  & = \int_{\cont(u)} \duality{z(t) - u(t) + y(t)}{\de\D y(t)} \notag \\
  & \phantom{= \ } + 
          \sum_{t \in \discont(u)} \duality{z(t) - u(t+) + y(t+)}{y(t+) - y(t)} \le 0 \notag
\end{align}
and \eqref{BV-l int ineq for P} is proved.
\end{proof}

We conclude with the 

\begin{proof}[Proof of Theorem \ref{main-right continuous case}]
We proceed as in the previous proof but we use now \eqref{BV jump cond for P-}: if $t \in \discont(u)$ it reads 
\begin{equation}
  \duality{z - u(t) + y(t)}{y(t) - y(t-)} \le 0 \qquad \forall \zeta \in \Z.
\end{equation}
Hence, since $\cont(y) = \cont(u)$ and $u$ is right continuous, for every $z \in \L^\infty(\clint{0,T};\H)$ with 
$z(\clint{0,T}) \subseteq \Z$ we have
\begin{align}
  & \int_{\clint{0,T}} \duality{z(t) - u(t) + y(t)}{\de\D y(t)} \notag \\
  & = \int_{\cont(u)} \duality{z(t) - u(t) + y(t)}{\de\D y(t)} \notag \\
  & \phantom{= \ } + 
          \sum_{t \in \discont(u)} \duality{z(t) - u(t) + y(t)}{y(t) - y(t-)} \le 0 \notag
\end{align}
and we are done.
\end{proof}


\begin{Ack}
The author is grateful to the referees, whose comments and remarks led to improve the paper. 
\end{Ack}



\end{document}